\documentclass[11pt]{article}
\usepackage{amsthm,amssymb,amsmath, amsfonts}
\usepackage{enumerate}
\usepackage{hyperref}
\hypersetup{
    pdftitle=   {An upper bound on tricolored ordered sum-free sets},
   pdfauthor=  {Sang-il Oum}
}
\usepackage{authblk}
\newtheorem{THM}{Theorem}
\newtheorem{LEM}{Lemma}

\newcommand\abs[1]{\lvert #1\rvert}
\newcommand\F{\mathbb F}
\begin{document}
\title{An upper bound on tricolored ordered sum-free sets}
\author{Taegyun Kim\thanks{\texttt{ktg11k@kaist.ac.kr}} }
\author{Sang-il Oum\thanks{\texttt{sangil@kaist.edu}}
\thanks{This work was supported by the National Research Foundation of Korea (NRF) grant funded by the Korea government (MSIT) (No. NRF-2017R1A2B4005020).}}
\affil{Department of Mathematical Sciences, KAIST \\ Daejeon, South Korea.}
\date\today
\maketitle

\begin{abstract}
  We present a strengthening of the lemma on the lower bound of the
  slice rank by Tao~\cite{Tao2016} 
  motivated by the Croot-Lev-Pach-Ellenberg-Gijswijt bound on cap
  sets~\cite{CLP2016,EG2017}.
  The Croot-Lev-Pach-Ellenberg-Gijswijt method and 
  the lemma of Tao are based on the fact that the rank of a diagonal
  matrix is equal to the number of non-zero diagonal entries.
  Our lemma is based on the rank of upper-triangular matrices.
  This stronger lemma allows us to prove the following extension 
  of the Ellenberg-Gijswijt result \cite{EG2017}.
  A \emph{tricolored ordered sum-free set} in $\F_p^n$
  is a collection $\{(a_i,b_i,c_i):i=1,2,\ldots,m\}$ of ordered triples in $(\F_p^n )^3$ such that 
  $a_i+b_i+c_i=0$
  and if $a_i+b_j+c_k=0$, then $i\le j\le k$.
  By using the new lemma, we present an upper bound on the size of a tricolored ordered sum-free set in $\mathbb F_p^n$.
%  This generalizes the result of Kleinberg~\cite{Kleinberg2016}, who showed that 
%the same bound holds for a tricolored ordered sum-free set in $\mathbb F_p^n$, that is a collection  $\{(a_i,b_i,c_i):i=1,2,\ldots,m\}$ of ordered triples in $(\mathbb F_p^n)^3$ such that 
%  $a_i+b_j+c_k=0$ if and only if  $i= j= k$.
\end{abstract}
\section{Introduction}
Let $\F$ be a field.  A function $f:A^k\to \F$ is called a \emph{slice} if it can be written in the form
\[f(x_1,x_2,\ldots,x_k)=h(x_i) g(x_1,\ldots,x_{i-1},x_{i+1},\ldots,x_k)\]
for functions $h:A\to \F$ and $g:A^{k-1}\to \F$ with some $1\le i\le k$. 
The \emph{slice rank} of a function $f: A^k\to \F$, introduced by Tao~\cite{Tao2016}, is the minimum $k$ such that $f$
can be written as a sum of $k$ slices. 
If $k=2$, then the slice rank is equivalent to the usual concept of the matrix rank.
Then he showed the following:
\begin{LEM}[Tao~\cite{Tao2016}]\label{lem:tao}
  Let $A$ be a finite set.
  Let $f:A^k\to \F$  be  a function such that 
  $f(x_1,x_2,\ldots,x_k)\neq 0$ implies that $x_1=x_2=\cdots=x_k$.
  Then the slice rank of $f$ is equal to 
  $\abs{\{x\in A: f(x,x,\ldots,x)\neq 0\}}$.
\end{LEM}

This lemma was formulated from the proofs of the recent breakthrough on
cap sets by Croot, Lev, and Pach~\cite{CLP2016} and Ellenberg and
Gijswijt~\cite{EG2017}. This powerful new method led many results in
extremal combinatorics within a short period of time.

Note that when $k=2$, Lemma~\ref{lem:tao} is about the rank of diagonal matrices and it is immediate that the rank of the diagonal matrices is equal to the number of non-zero diagonal entries.
Then it is natural to wonder whether there is any formulation to use upper-triangular matrices as a basis step. Here is such a generalization.
\begin{LEM}\label{lem:main}
  Let $\preceq$ be a linear ordering of $A$.
  Let $f:A^k\to \F$  be  a function such that 
  $f(x_1,x_2,\ldots,x_k)\neq 0$ implies that $x_1\preceq x_i\preceq x_k$ for all $i=2,3,\ldots,k-1$.
  Then the slice rank of $f$ is at least
  \[\abs{\{x\in A: f(x,x,\ldots,x)\neq 0\}}.\]
\end{LEM}
Though the proof based on the induction is very similar to the proof of Lemma~\ref{lem:tao} by Tao, we present its proof in Section~\ref{sec:proof}.
As Lemma~\ref{lem:main} includes Lemma~\ref{lem:tao},
it implies all other results previously proven by using Lemma~\ref{lem:tao}.
In the next section, we present an application of the new lemma.
We hope to find further interesting applications.

\section{Application: Tricolored ordered sum-free sets}\label{sec:appl}
We present one application of this new lemma. 
Blasiak~et al.~\cite{BCCGNSU2017} and independently Alon (in \cite{BCCGNSU2017})
observed that the result of Ellenberg and Gijswijt~\cite{EG2017} can
be extended to tricolored sum-free sets in $\F_p^n$.
A \emph{tricolored sum-free set} in an abelian group $H$ is a
collection $\{(a_i,b_i,c_i)\}_{i=1}^m$ of ordered triples in $H$ such
that $a_i+b_j+c_k=0$ if and only if $i=j=k$.
We will extend it further to tricolored \emph{ordered} sum-free sets.

A \emph{tricolored ordered sum-free set} in an abelian group $H$ is a collection $\{(a_i,b_i,c_i)\}_{i=1}^m$ of ordered triples in $H$ such
that
\begin{enumerate}[(i)]
\item $a_i+b_i+c_i=0$ for all $i=1,2,\ldots,m$,
\item if $a_i+b_j+c_k=0$, then  $i\le j\le k$.
\end{enumerate}
We remark that a tricolored sum-free set is a tricolored ordered sum-free set such that $a_i+b_j+c_k=0$ if and only if $i=j=k$. 
We prove that the same upper bound can be achieved for tricolored ordered
sum-free sets of $\F_p^n$, as it was done for 
cap sets of $\F_p^n$ (Ellenberg and Gijswijt~\cite{EG2017}, see
Tao~\cite{Tao2016})
and for  
tricolored sum-free sets
of $\F_p^n$ by Blasiak~et al.~\cite{BCCGNSU2017} and independently Alon (in \cite{BCCGNSU2017}).
\begin{THM}
  If $\{(a_i,b_i,c_i)\}_{i=1}^m$ is a tricolored ordered sum-free set in $\F_p^n$, then $ m\le 3N$
  where $N$ is the number of monomials of total degree at most $(p-1)n/3$
  and in which each variable has degree at most $p-1$.

  In other words, \[N=\sum \frac{n!}{n_0!n_1!\cdots n_{p-1}!}\]
  where the sum is taken over all non-negative integers $n_0$,$n_1$, $\ldots$, $n_{p-1}$ such that $ n_0+n_1+\cdots+n_{p-1}=n$ and 
  $n_1+2n_2+\cdots+(p-1)n_{p-1}\le (p-1)n/3 $.
\end{THM}
\begin{proof}
  Let $A=\{1,2,\ldots,m\}^n$.
  and 
  let
  \[f(x,y,z)= \prod_{\ell=1}^n \left( 
      1- (a_{x_\ell} +b_{y_\ell}+c_{z_\ell})^{p-1}
      \right).\]
  Then $f(x,x,x)=1$ for all $x\in A$
  and if $f(x,y,z)\neq 0$, then $x\le y\le z$.
  By Lemma~\ref{lem:main}, the slice rank of $f$ is at least $m$.

  Now let us show that the slice rank of $f$ is at most $3N$.
  The next steps are now routine, as it is done in Tao~\cite{Tao2016}.
  We expand $f$ as 
  \[
    f(x,y,z)= \sum_{i,j,k} a_{x_1}^{i_1} \cdots a_{x_n}^{i_n}
    b_{y_1}^{j_1} \cdots b_{y_n}^{j_n}
    c_{z_1}^{k_1} \cdots c_{z_n}^{k_n}\]
  and collect terms based on whether $i_1+\cdots+i_n\le (p-1)/3$,
  $j_1+\cdots+j_n\le (p-1)/3$, or
  $k_1+\cdots+k_n\le (p-1)/3$.
  Note that $N$ is equal to the number of tuples $(i_1,i_2,\ldots,i_n)$
  of non-negative integers such that $i_1+i_2+\cdots+i_n\le 3(p-1)/n$
  and $i_1,i_2,\ldots,i_n\le p-1$.
  
  Then $f$ can be written as a sum of at most $3N$ slices, 
  where each slice has a term $a_{x_1}^{i_1} \cdots a_{x_n}^{i_n}$
  for $i_1+\cdots +i_n\le (p-1)/3$ and $0\le i_1,i_2,\ldots,i_n\le p-1$,
or a term  $b_{y_1}^{j_1} \cdots b_{y_n}^{k_n} $
  for $j_1+\cdots +j_n\le (p-1)/3$ and $0\le j_1,j_2,\ldots,j_n\le
  p-1$, 
or a term $ c_{z_1}^{k_1} \cdots c_{z_n}^{k_n}$
for  $k_1+\cdots+k_n\le (p-1)/3$ and $0\le k_1,k_2,\ldots,k_n\le p-1$.
Hence the slice rank of $f$ is at most $3N$
and so $m\le 3N$.
\end{proof}

\section{Proof of the lemma}\label{sec:proof}
Here we present the proof of our new lemma.
\begin{proof}[Proof of Lemma~\ref{lem:main}]
  We proceed by induction on $k$. 
  If $k=2$, then the slice rank of $f$ is equal to the rank of the corresponding matrix, which is upper triangular 
  and the conclusion follows trivially.
  
  Thus we may assume $k>2$.
  We may also assume that $f(x,x,\ldots,x)\neq0$ for all $x\in A$ because otherwise we can discard such $x$ from $A$. 

  Suppose that the slice rank of $f$ is less than $\abs{A}$. Then there exist disjoint sets $I_1$, $I_2$, $\ldots$, $I_{k}$ of indices 
  and functions $f_{i,\alpha}:A\to \F$ 
  and $g_{i,\alpha}:A^{k-1}\to \F$ for $\alpha\in I_i$
  such that
  $\sum_{i=1}^k \abs{I_i}<\abs{A}$ and 
  \[
    f(x_1,\ldots,x_k)=\sum_{i=1}^k \sum_{\alpha\in I_i} f_{i,\alpha}(x_i) 
    g_{i,\alpha} (x_1,\ldots,x_{i-1},x_{i+1},\ldots,x_k).
  \]
  Let $W$ be a vector space of functions $h:A\to \F$ such that 
  \[\sum_{a\in A} f_{2,\alpha}(a)h(a)=0\]  for all $\alpha\in I_2$.
  Let  $d=\dim W$. Then $d\ge \abs{A}-\abs{I_2}$.
  Let $B$ be a basis of $W$. Then there exists a subset $A'$ of $A$ such that 
  $\abs{A'}=d$ 
  and functions in $B$ restricted on $A'$ are linearly independent.
  Then every function from $A'$ to $\F$ can be extended to a function in $W$
  and therefore there exists a function $h\in W$ such that $h(a)=1$ for all $a\in A'$.
  
  Then 
  \begin{multline*}
    \sum_{x_2\in A} f(x_1,\ldots,x_k) h(x_2)
    =\\ \sum_{i\neq 2} \sum_{\alpha\in I_i}
    f_{i,\alpha}(x_i)\sum_{x_2\in A} g_{i,\alpha} (x_1,\ldots,x_{i-1},x_{i+1},\ldots,x_k) h(x_2).
    \end{multline*}
  Let $f'(x_1,x_3,x_4,\ldots,x_k)= \sum_{x_2\in A} f(x_1,\ldots,x_k) h(x_2)$.
  It is easy to observe that if $f'(x_1,x_3,\ldots,x_k)\neq 0$, then 
  $x_1\preceq x_i\preceq x_k$ for all $i=3,4,\ldots,k-1$.
  Furthermore for each $x\in A'$,  $f'(x,x,\ldots,x)=\sum_{x_2\in A} f(x,x_2,x,x,\ldots,x)h(x_2)
^n  =f(x,x,\ldots,x) h(x)\neq 0$.
  Here we use the assumption that $f(x,x_2,x,\ldots,x)\neq 0$ implies 
  $x\preceq x_2\preceq x$ and $h(x)=1$ for all $x\in A'$.
  Therefore by the induction hypothesis, 
  the slice rank of $f'$ is at least $\abs{A'}$.
  Currently $f'$ is written as a sum of $\abs{I_1}+\abs{I_3}+\cdots+\abs{I_k}$ slices and so 
  \[
    d=\abs{A'}\le \abs{I_1}+\abs{I_3}+\abs{I_4}+\cdots+\abs{I_k}.
  \]
  Then  $\abs{A}\le \abs{I_1}+\abs{I_2}+\cdots+\abs{I_k}$, contradicting the hypothesis.
\end{proof}

\providecommand{\bysame}{\leavevmode\hbox to3em{\hrulefill}\thinspace}
\providecommand{\MR}{\relax\ifhmode\unskip\space\fi MR }
\providecommand{\MRhref}[2]{%
  \href{http://www.ams.org/mathscinet-getitem?mr=#1}{#2}
}
\providecommand{\href}[2]{#2}

%\bibliographystyle{amsplain}
%\bibliography{slicerank}
\end{document}